\documentclass[11pt]{amsart}

\usepackage{palatino}
\usepackage{amsfonts}
\usepackage{amssymb}
\usepackage{amscd}
\usepackage{xypic}
\usepackage[breaklinks,bookmarksopen,bookmarksnumbered]{hyperref}
\usepackage{graphicx}
\usepackage{float}

\newtheorem{theorem}{Theorem}[section]
\newtheorem{proposition}[theorem]{Proposition}
\newtheorem{corollary}[theorem]{Corollary}
\newtheorem{lemma}[theorem]{Lemma}
\theoremstyle{definition}

\newtheorem{remark}[theorem]{Remark}

\newtheorem{conjecture/question}[theorem]{Conjecture/Question}
\newtheorem{question}[theorem]{Question}
\newtheorem{remark/definition}[theorem]{Remark/Definition}
\newtheorem{terminology/notation}[theorem]{Terminology/Notation}

\def\PP{{\textbf P}}

\def\OO{\mathcal{O}}

\def\cD{\mathcal{D}}

\def\F{\mathcal{F}}

\def\I{\mathcal{I}}

\def\cM{\mathcal{M}}
\def\cR{\mathcal{R}}

\def\rr{\overline{\mathcal{R}}}

\def\mm{\overline{\mathcal{M}}}

\def\dd{\overline{\mathcal{D}}}

\DeclareFontFamily{OT1}{pzc}{}
\DeclareFontShape{OT1}{pzc}{m}{it}{<-> s * [1.10] pzcmi7t}{}
\DeclareMathAlphabet{\mathpzc}{OT1}{pzc}{m}{it}

\begin{document}
\title{Secant loci on moduli of Prym varieties}

\author[G. Farkas]{Gavril Farkas}

\address{Humboldt-Universit\"at zu Berlin, Institut f\"ur Mathematik,  Unter den Linden 6
\hfill \newline\texttt{}
 \indent 10099 Berlin, Germany} \email{{\tt farkas@math.hu-berlin.de}}
\thanks{We are grateful to Alessandro Verra for many discussions related to this circle of ideas.}

\author[M. Lelli-Chiesa]{Margherita Lelli-Chiesa}
\address{Universit\`a Roma Tre, Dipartimento di Matematica, Largo San Leonardo Murialdo \hfill
 \newline \indent 1-00146 Roma, Italy}
 \email{{\tt
margherita.lellichiesa@uniroma3.it}}

\begin{abstract}
We present a Prym analogue of Lazarsfeld's result that curves on general polarized K3 surfaces verify the Brill-Noether Theorem, or equivalently, that their canonical embedding has no unexpected secants. We show that the Prym-canonical embedding of a curve on a  general Nikulin surface (both of standard and non-standard types) has no unexpected secants and explain how these two geometric facts suffice to determine the class of the difference divisor on the moduli space $\rr_g$ of stable Prym curves of (odd) genus $g$.
\end{abstract}

\maketitle

\section{introduction}

The Brill-Noether Theorem, asserting that for a general curve $C$ of genus $g$ the varieties $W^r_d(C):=\{L\in \mbox{Pic}^d(C): h^0(C,L)\geq r+1\}$ have the expected dimension $$\rho(g,r,d)=g-(r+1)(g-d+r)$$ is one of the cornerstones of the theory of algebraic curves, having multiple applications to enumerative geometry, or to the birational geometry of the moduli space $\mm_g$ of stable curves of genus $g$, see \cite{EH}. Via Riemann-Roch, the Brill-Noether Theorem can be reformulated in terms of secant varieties to the canonical linear system of $C$. More generally, given a line bundle $L\in \mbox{Pic}^d(C)$ with $h^0(C,L)=r+1$ and positive integers $0\leq f< e$, one introduces the variety 
\begin{equation}\label{eq:sec_var1}
V_e^{e-f}(L):=\bigl\{Z\in C_e: h^0(C, L(-Z))\geq r+1-e+f\bigr\}
\end{equation}
of effective divisors $Z$ of degree $e$ which fail to impose $f$ independent conditions on $|L|$. Since $V_e^{e-f}(L)$ can be represented as a degeneracy locus on the $e$-th symmetric product $C_e$ of $C$ globalizing the evaluation maps $\mathrm{ev}_Z \colon H^0(C,L) \rightarrow L_{D}$, one obtains that every irreducible component of $V_e^{e-f}(L)$ has dimension at least equal to $e-f(r+1-e+f)$. With this terminology, observing that $V_d^{d-r}(\omega_C)\cong C_d^r \rightarrow W^r_d(C)$ is generically a $\PP^r$-bundle, the Brill-Noether Theorem is equivalent to the statement that for a general curve $C$ of genus $g$, all varieties $V_e^{e-f}(\omega_C)$ have the expected. In particular, Lazarsfeld's Theorem \cite{La1} can then be formulated as saying that if $X$ is a smooth $K3$ surface with $\mbox{Pic}(X)=\mathbb Z\cdot H$, where $H^2=2g-2$, then for any irreducible curve $C\in |H|$ the secant loci of $\omega_C$ have the expected dimension. For further statements on the applications of these loci to the Green-Lazarsfeld Secant Conjecture on syzygies of curves we refer to \cite{FK}, or \cite{F}.

\vskip 4pt

Moving to the Prym setting, let $\cR_g$ be the moduli space parameterizing pairs $[C, \eta]$, where $C$ is a smooth curve of genus $g$ and $\eta\in \mbox{Pic}^0(C)$ is a $2$-torsion point. The Prym moduli space $\cR_g$ admits a Deligne-Mumford compactification $\rr_g:=\mm_g\bigl(\mathcal{B} \mathbb Z_2\bigr)$ constructed by Cornalba \cite{Cor}, whose birational geometry has been described in detail in \cite{FL}. In particular, $\rr_g$ has negative Kodaira dimension for $g\leq 9$, see \cite{Don}, \cite{FV1}, \cite{FV2}, \cite{FV3}, and it is a variety of general type for $g\geq 13$ and $g\neq 16$, see \cite{FL}, \cite{CEFS}, \cite{FJP} and \cite{Br}. For odd genus $g=2i+1$ a key role in all these results is played by the (closure in $\rr_{2i+1}$ of the) \emph{difference Prym divisor}
\begin{equation}\label{eq:diff_div}
\cD_{2i+1}:=\Bigl\{[C, \eta]\in \cR_{2i+1}: \eta\in C_i-C_i\Bigr\}.
\end{equation}
The defining condition for points in $\cD_{2i+1}$ can be reformulated as $V_i^{i-1}(\omega_C\otimes \eta)\neq \emptyset$, or equivalently, that the Prym-canonical embedding 
$\varphi_{\omega_C\otimes \eta}\colon C\hookrightarrow \PP^{2i-1}$
possesses an $i$-secant $(i-2)$-plane. In particular, $\cD_{2i+1}$ becomes the exact Prym analogue of the \emph{Hurwitz divisor} 
\begin{equation}\label{eq:hurwitz_div}
\cM_{2i+1,i+1}^1:=\Bigl\{[C]\in \cM_{2i+1}: V_{i+1}^i(\omega_C)\neq \emptyset\Bigr\}=
\Bigl\{[C]\in \cM_{2i+1}:\exists \ C\stackrel{i+1:1}\longrightarrow \PP^1 \Bigr\}
\end{equation}
defined by Harris and Mumford \cite{HM} and instrumental in showing that 
$\mm_g$ is of general type for odd $g>23$.  

\vskip 3pt

The first result we establish determines the dimension of the secant loci for the general Prym-canonical curve of genus $g$.

\begin{theorem}\label{thm:genpry}
We fix integers $0\leq f< e< g$. Then for a general Prym curve $[C, \eta]\in \cR_g$ the secant locus $V_e^{e-f}(\omega_C\otimes \eta)$ is equidimensional of dimension 
$$\mathrm{dim} \ V_e^{e-f}(\omega_C\otimes \eta)=e-f(g-1-e+f).$$
In particular, if $e-f(g-1-e+f)<0$, then $V_e^{e-f}(\omega_C\otimes \eta)=\emptyset$.
\end{theorem}

The proof of Theorem \ref{thm:genpry} uses degeneration and is presented in Section \ref{sec:degprym}. In what follows, we discuss a more refined version of Theorem \ref{thm:genpry} for Prym curves on a special class of $K3$ surfaces.

\vskip 3pt

\subsection{Nikulin surfaces and Prym varieties.}  Lazarsfeld's result \cite{La1} provides explicit examples of smooth curves lying on general $K3$ surfaces which satisfy the Brill-Noether Theorem, in particular, which lie outside the Hurwitz 
divisor $\mm_{2i+1,i+1}^1$. This geometric condition alone is sufficient to determine the slope $s\bigl([\mm_{2i+1,i+1}^1]\bigr)=6+\frac{12}{g+1}$ of the Hurwitz divisor. We show that in Prym setting, one has a similar result for curves lying on Nikulin surfaces and, in the process, we obtain a more precise version of Theorem \ref{thm:genpry}.

\vskip 4pt

A \emph{Nikulin  surface} is a $K3$ surface $X$ endowed with a  double cover
$$f\colon \tilde X \longrightarrow X $$
branched over $8$ disjoint smooth rational curves $N_1, \ldots, N_8\subseteq X$. Blowing down the $(-1)$-curves $f^{-1}(N_j)$, one obtains a minimal $K3$ surface 
$\sigma\colon \tilde{X}\rightarrow Y$, together with an involution $\iota\in \mbox{Aut}(Y)$ having $8$ fixed points corresponding to the images of the curves $\sigma(f^{-1}(N_j))$. The class $\mathcal O_X(N_1+\cdots+N_8)$ is divisible by $2$ and we set $\mathfrak{e} \in \mbox{Pic}(X)$ to be the class characterized by $\mathfrak{e}^{\otimes 2}\cong \OO_X(N_1+\cdots+N_8)$. 

\vskip 4pt

Let $\mathfrak{N}$ denote the Nikulin lattice defined as the even lattice of rank $8$ generated by the classes of $N_1, \ldots, N_8$ and by $\mathfrak{e}$. A \emph{polarized Nikulin surface} of genus $g$ is a $K3$ surface $X$ together with an embedding $\mathbb Z\cdot L\oplus_{\perp} \mathfrak{N}\hookrightarrow \mbox{Pic}(X)$, such that $L$ is a big and nef class on $X$, with $L^2=2g-2$. If this embedding is primitive, we say that $X$ is a \emph{standard} Nikulin surface, else we say that $X$ is \emph{non-standard}\footnote{Occasionally standard (respectively non-standard) Nikulin surfaces are referred to as being of the \emph{first} (respectively of the \emph{second} kind), see \cite{FV1}, \cite{FV2}.}.
There are two irreducible $11$-dimensional moduli spaces $\F_g^{\mathfrak{N}}$ respectively $\mathcal{F}_g^{\mathfrak{N}, \mathrm{ns}}$ of standard (respectively non-standard) polarized Nikulin surfaces of genus $g$, see \cite{Dol}, \cite{vGS} for details. Over $\F_g^{\mathfrak{N}}$ we consider the $\PP^g$-bundle of pairs
$$\mathcal{P}_g^{\mathfrak{N}}:=\Bigl\{(X, C): C\subseteq X \mbox{ is a smooth curve such that }\  \bigl[X, \mathfrak{N}\oplus_{\perp} \OO_X(C)\bigr]\in \F_g^{\mathfrak{N}}\Bigr\},$$
together with forgetful maps
$$\xymatrix{
  & \mathcal{P}_g^{\mathfrak{N}} \ar[dl]_{p_g} \ar[dr]^{\chi_g} & \\
   \F_g^{\mathfrak{N}} & & \cR_{g}       \\
                 }$$
where $p_g\bigl([X,C]\bigr):=[X, \OO_X(C)]$ and $\chi_g\bigl([X,C]\bigr):=[C, \mathfrak{e}_C^{\vee}:=
\mathfrak{e}^{\vee}\otimes \OO_C]$. In particular, $\chi_g$ can be used to find a uniform parametrization of $\cR_g$ for small $g$, see \cite[Theorem 0.2]{FV1}. Keeping this notation throughout, we can now state our next result:

\begin{theorem}\label{thm:disjoint1}
Let $[X, \mathfrak{N}\oplus_{\perp} \mathbb Z\cdot L]\in \F_g^{\mathfrak{N}}$ be a general standard Nikulin surface of genus $g$. Then for every irreducible curve $C\in |L|$, we have that 
 $V_e^{e-f}(\omega_C\otimes \mathfrak{e}_C)=\emptyset$, whenever $e-f(g-1-e+f)<0$. In particular, for odd genus $g=2i+1$, we have that $\mathfrak{e}_C\notin C_i-C_i$. 
\end{theorem}

Theorem \ref{thm:disjoint1} thus provides the first examples of explicit smooth Prym curves of odd genus which are generic from the point of view of the Bril-Noether theory of the Prym-canonical system.

\vskip 4pt

If standard Nikulin surfaces are instrumental in parametrizing 
$\rr_g$, as explained in \cite{KLV1}, non-standard Nikulin surfaces can be used to parametrize the boundary divisor $\Delta_0^{\mathrm{ram}}$ of $\rr_g$ for low $g$. Precisely, if $\Delta_0$ is the boundary divisor of $\mm_g$ parametrizing irreducible nodal curves of genus $g$ and $\pi\colon \rr_g\rightarrow \mm_g$ is the forgetful morphism dropping the Prym structure, then 
$$\pi^*(\Delta_0)=\Delta_0'+\Delta_0''+2\Delta_0^{\mathrm{ram}},$$
see \cite[1]{FL}. In particular $\Delta_0^{\mathrm{ram}}$ is the ramification divisor of $\pi$ and its general point corresponds to a Prym curve $[X=C\cup_{\{x,y\}} \PP^1, \eta]\in \rr_g$, where $C$ is a smooth curve of genus $g-1$ meeting a smooth rational component $\PP^1$ at two points $x, y$, whereas $\eta\in \mbox{Pic}^0(X)$ satisfies $\eta_{\PP^1}\cong \OO_{\PP^1}(1)$ and $\eta_C^{\otimes 2}\cong \OO_C(-x-y)$. In other words, the line bundle $\eta_C$ induces a degree $2$ cover of $C$ branched over $x$ and $y$. 

\vskip 4pt

Non-standard Nikulin surfaces appear only for odd genus. If $X$ is a non-standard Nikulin surface of genus $g\equiv 3 \mbox{ mod } 4$, then following \cite[Proposition 2.2]{vGS}, one can show that up to a permutation of the $(-2)$-curves $N_1, \ldots, N_8$, there exist classes $R, R'$ such that 
\begin{equation}\label{eq:R_and_R'}
R \equiv \frac{L-N_1-N_2}{2}, \ \  \   \mbox{  } \ \ R'\equiv \frac{L-N_3-\cdots-N_8}{2}\in \mbox{Pic}(X).
\end{equation}
Moreover, for a general point of $\F_g^{\mathfrak{N}, \mathrm{ns}}$, one has that $\mbox{Pic}(X)\cong \mathbb Z\cdot R\oplus \mathfrak{N}$.

\vskip 4pt

We now set $L^2=16i-4$, therefore curves in the linear system $|L|$ have genus $8i-1$. Using (\ref{eq:R_and_R'}) a smooth curve $C\in |R|$ has genus $g= 2i$. Moreover $C\cdot N_1=C\cdot N_2=1$, while $C\cdot N_i=0$ for $i=3, \ldots, 8$. It follows that $f_{|f^{-1}(C)}\colon f^{-1}(C)\rightarrow C$ is a double cover ramified over two points $x_1=C\cdot N_1$ and $x_2=C\cdot N_2$.
Equivalently,
$$\xi(C):=\bigl[C\cup_{\{x_1, x_2\}} \PP^1,\ \eta_C:=\mathfrak{e}_C^{\vee}, \   \eta_{\PP^1}\cong \OO_{\PP^1}(1)\bigr]\in \Delta_0^{\mathrm{ram}}\subseteq \rr_{2i+1}.$$
One obtains a rational map $\xi\colon |R| \dashrightarrow \rr_{2i+1}$
and for a Lefschetz pencil $\Xi \cong \PP^1\subseteq |R|$, a corresponding pencil of Prym curves $\Xi^{\mathrm{ns}}:=\xi(\Xi)$ of genus $2i+1$. 

\vskip 4pt

We present an analogue of Theorem \ref{thm:disjoint1} in the setting of non-standard Nikulin surfaces.

\begin{theorem}\label{thm:disjoint2}
Let $[X, \mathfrak{N}\oplus_{\perp} \mathbb Z\cdot L]\in \F_{8i-1}^{\mathfrak{N}, \mathrm{ns}}$ be a general non-standard Nikulin surface of genus $8i-1$. Then for every irreducible nodal curve $C\in |R|$, we have that $\xi\bigl([C]\bigr)\notin \overline{\cD}_{2i+1}$.  
\end{theorem}

Thus, also non-standard Nikulin surfaces provide explicit examples of Prym curves that are general from the point of view of Prym-Brill-Noether theory. Theorems \ref{thm:disjoint1} and \ref{thm:disjoint2} provide two explicit pencils of Prym curves on Nikulin surfaces which are disjoint from the closure $\dd_{2i+1}$ of the difference Prym divisor in $\rr_{2i+1}$. These two properties uniquely determine the slope of $[\dd_{2i+1}]\in \mbox{Pic}(\rr_{2i+1})$, as we shall explain.  

\vskip 3pt

We denote by $\lambda\in \mbox{Pic}(\rr_g)$ the Hodge class and by $\delta_0':=[\Delta_0'], \ \delta_0'':=[\Delta_0'']$ and respectively $\delta_0^{\mathrm{ram}}:=[\Delta_0^{\mathrm{ram}}]$ the corresponding boundary divisor classes, see also \cite[1]{FL} for background on the Picard group of $\rr_{g}$. Then if 
$\Xi^{\mathrm{s}}\subseteq \rr_{2i+1}$ and $\Xi^{\mathrm{ns}}\subseteq \rr_{2i+1}$ denote the pencils associated to a standard (respectively non-standard) Nikulin surface of genus $2i+1$ (respectively $8i-1$), then we have the following relations (see Proposition \ref{prop:intnumb}): 

\begin{align*}
\Xi^{\mathrm{s}}\cdot \lambda=2i+2, \  \  \Xi^{\mathrm{s}}\cdot \delta_0'=12i+8, \  \ \ \Xi^{\mathrm{s}}\cdot \delta_0''=0\   \ \mbox{ and } \  \Xi^{\mathrm{s}}\cdot \delta_0^{\mathrm{ram}}=8,\\
\Xi^{\mathrm{ns}}\cdot \lambda= 2i+1, \  \ \Xi^{\mathrm{ns}}\cdot \delta_0'=12i+6, \ \ \  \Xi^{\mathrm{ns}}\cdot \delta_0''=0 \ \mbox{ and  }\ \Xi^{\mathrm{ns}}\cdot \delta_0^{\mathrm{ram}}=4.
\end{align*}
Imposing the condition that the intersection numbers of $[\dd_{2i+1}]$ with both pencils $\Xi^{\mathrm{s}}$ and $\Xi^{\mathrm{ns}}$ are equal to zero (which amounts to a slight refinement of both Theorems \ref{thm:disjoint1} and \ref{thm:disjoint2}), we rederive one of the main results of \cite{FL}:

\begin{corollary}\label{cor:class}
The class $[\dd_{2i+2}]$ is given up to a positive constant by the following formula:
$$[\dd_{2i+2}]=c\Bigl((3i+1)\lambda-\frac{i}{2}\bigl(\delta_0'+\delta_0''\bigr)-\frac{2i+1}{4}\delta_0^{\mathrm{ram}}-\cdots\Bigr)\in \mathrm{Pic}(\rr_{2i+1}).$$
\end{corollary}

The class $[\dd_{2i+2}]$ has numerous desirable features, above all, its very low slope in the $(\lambda, \delta_0')$-plane of $\mbox{Pic}(\rr_{2i+1})$. As such,  the divisor $\dd_{2i+1}$ is used in the proof that $\rr_{g}$ is of general type for odd $g\geq 13$.

\vskip 5pt

\noindent {\bf Acknowledgments:} Farkas was supported by the Berlin Mathematics Research Center MATH+ and by the ERC Advanced Grant SYZYGY.
This project has received funding from the European Research Council (ERC) under the European Union Horizon 2020 research and innovation program (grant agreement No. 834172). Lelli-Chiesa was partially supported by PRIN2022 “Moduli spaces and Birational Geometry” and by the INdAM group GNSAGA.

\section{Secant divisors of general Prym-canonical curves}\label{sec:degprym}

For a smooth curve $C$ and an integer $e>0$, we denote by $C_e$ its $e$-th symmetric product. We recall the definition of the varieties of secant divisors to a linear system on a curve $C$. We fix  integers $0\leq f< e$, a line bundle $L\in \mbox{Pic}^d(C)$ and set $r:=h^0(C,L)-1$. Let us introduce the universal degree $e$ effective divisor $$\mathcal{Z}:=\Bigl\{(Z,x)\in C_e\times C: x\in \mathrm{supp}(Z)\Bigr\}$$ and let $\pi_1\colon C_e\times C\rightarrow C_e$ respectively 
$\pi_2\colon C_e\times C\rightarrow C$ be the two projections. We define the rank $e$ tautological bundle $E_L=L^{[e]}:=(\pi_1)_*\bigl(\pi_2^*(L)\otimes \OO_{\mathcal{Z}}\bigr)$ having fibres $E_{L|Z}\cong L_{D}=L\otimes \OO_Z$ over a point $Z\in C_e$ and consider the evaluation morphism 
$$\mathrm{ev}\colon H^0(C, L)\otimes \OO_{C_e}\longrightarrow E_L.$$ Then the variety of secant divisors $V_e^{e-f}(L)$ defined in (\ref{eq:sec_var1}) is  defined as the degeneracy locus of those divisors $Z\in C_e$ for which $\mbox{rk}\bigl(\mathrm{ev}_Z\bigr)\leq e-f$. The expected dimension $\mbox{exp. dim }  V_e^{e-f}(L)$ is then accordingly equal to $e-f(r+1-e+f)$.

\vskip 3pt

The study of the loci $V_e^{e-f}(\omega_C)$ is by definition equivalent to that of the Brill-Noether loci $W^r_d(C)$. For line bundles $L$ of large degree, the Green-Lazarsfeld \emph{Secant Conjecture} predicts a close relation between the non-emptiness of the varieties $V_e^{e-f}(L)$ and the non-vanishing of the Koszul cohomology groups $K_{p,q}(C,L)$ and we refer to \cite{FK} and \cite{F} for new developments and details. 

\vskip 4pt

We take a Prym curve $[C, \eta]\in \cR_g$ and let $L:=\omega_C\otimes \eta$ be the Prym-canonical linear system. We fix integers $0\leq f< e< g$. By Riemann-Roch, we then have the following equivalence 
$$V_e^{e-f}(\omega_C\otimes \eta)\neq \emptyset\Longleftrightarrow \eta\in C_e-W_e^{f-1}(C)\subseteq \mbox{Pic}^0(C).$$

\vskip 5pt

We prove Theorem \ref{thm:genpry} by using degeneration. First, we recall the notation for vanishing and ramification sequences of linear series on curves following \cite{EH}. If $\ell=(L,V)\in G^r_d(C)$ is a linear system on a smooth curve $C$, the \emph{ramification sequence} of $\ell$ at a point $q\in C$
\[
\alpha^{\ell}(q) : 0\leq \alpha_0^{\ell}(q) \leq \cdots \leq \alpha_r^{\ell}(q)\leq d-r
\]
is obtained from the \emph{vanishing sequence}
\[
a^{\ell}(q) : 0\leq a_0^{\ell}(q) < \cdots < a_r^{\ell}(q) \leq d
\]
by setting $\alpha^{\ell}_i(q) := a^{\ell}_i(q) - i$, for $i=0, \ldots, r$. The \emph{ramification weight} of $q$ with respect to $\ell$ is defined as the quantity $\mathrm{wt}^{\ell}(q) := \sum_{i=0}^r \alpha^{\ell}_i(q)$.

\vskip 4pt

A \emph{limit linear series} on a curve $C$  of compact type consists of a collection $$\ell=\Bigl\{\ell_Y=(L_Y, V_Y) \in G^r_d(Y): Y\mbox{ is a component of } C\Bigr\}$$ satisfying  compatibility conditions in terms of the vanishing sequences at the nodes of $C$ of the respective aspects, see \cite{EH}. We denote by $\overline{G}^r_d(C)$ the variety of limit linear series of type $g^r_d$ on $C$. More generally, if $q_1,  \ldots, q_s\in C_{\mathrm{req}}$ are smooth points and 
$$\alpha^i=\bigl(0\leq \alpha_0^i\leq \cdots \leq \alpha_r^i\leq d-r\bigr)$$ are \emph{Schubert indices}, we denote by $\overline{G}^r_d\bigl(C, (q_1,\alpha^1), \ldots, (q_s, \alpha^s)\bigr)$ the variety of limit linear series $\ell\in \overline{G}^r_d(C)$ satisfying the conditions $\alpha^{\ell}(q_i)\geq \alpha^i$ for $i=1, \ldots, s$. Each component of $\overline{G}^r_d\bigl(C, (q_1, \alpha^1), \ldots, (q_s, \alpha^s)\bigr)$ has dimension at least equal to $\rho(g,r,d)-\mbox{wt}(\alpha^1)-\cdots-\mbox{wt}(\alpha^s)$. 
Theorem 1.1 of \cite{EH} offers adequate sufficient conditions when the  actual dimension of the locus $\overline{G}^r_d\bigl(C, (q_1, \alpha^1), \ldots, (q_s, \alpha^s)\bigr)$ equals the expected dimension. 

\vskip 4pt

Our first result is the following:

\begin{theorem}\label{thm:0nonexistence}
Let $[C, \eta]\in \cR_g$ be a general Prym curve of genus $g$ and we fix integers $0\leq f<e<g$. If $e-f(g-1-e-f)<0$, then $V_e^{e-f}(\omega_C\otimes \eta)=\emptyset$.
\end{theorem}
\begin{proof}
Assume the conclusion to be false, in particular, for every $[C, \eta]\in \cR_g$, there exists $Z\in V_e^{e-f}(\omega_C\otimes \eta)$, therefore, by Riemann-Roch,  $\eta(Z)\in W^{f-1}_e(C)$. 

\vskip 3pt

We degenerate to nodal curves and fix general pointed elliptic tails $[E_j, x_j]\in \cM_{1,1}$ and non-trivial $2$-torsion points $\eta_j\in \mbox{Pic}^0(E_j)[2]$, for $j=1, \ldots, g$. We then consider the map 
\begin{equation}\label{eq:tilr}
j\colon \mm_{0,g}\longrightarrow \rr_g, \ \ [R, x_1, \ldots, x_g]\mapsto \bigl[C:=R\cup_{x_1} E_1\cup \ldots \cup_{x_g} E_g, \  \ \eta_{R}\cong \OO_R, \ \eta_{E_j}\cong \eta_j\bigr].
\end{equation}

By assumption, on each curve $C$ as above, after possibly inserting chains of smooth rational curves at the points $x_1, \ldots, x_g$, or at the nodes of $R$, there exists a refined limit linear series $\ell\in \overline{G}^{f-1}_e(C)$ and an effective divisor $Z$ of total degree $e$ on $C$, such that if $Z_j\subseteq Z$ denotes the (possibly empty) subdivisor consisting of all points of $Z$ lying on $E_j$, then the $E_j$-aspect of $\ell_{E_j}$ has $\eta_j\bigl(Z_j+(e-\mbox{deg}(Z_j))\cdot x_j\bigr)\in \mbox{Pic}^e(E_j)$ as its underlying line bundle, for $j=1, \ldots, g$. 

\vskip 4pt

Applying \cite[Proposition 2.2]{F2}, it follows that there exists \emph{one} degeneration to a flag curve $C$ like in (\ref{eq:tilr}), such that $R=R_1\cup_p R_2$, with both $R_1$ and $R_2$ being trees of smooth rational curves meeting at the point $p$ and there exists $0\leq m\leq e$ such that $x_1, \ldots, x_m\in R_1\setminus \{p\}$ and $x_{m+1}, \ldots, x_g\in R_2\setminus \{p\}$, and, moreover, the divisor $Z$ lies on $R_1\cup_{x_1} E_1\cup \ldots \cup_{x_m} E_m$. In other words, one always finds a degeneration to a flag curve of genus $g$ such that the points in the support of the divisor $Z$ specialize to a subcurve of genus $m\leq e$.

\vskip 4pt

Let $\ell\in \overline{G}^{f-1}_e(C)$ be the corresponding limit linear series and we denote by
$$\alpha^j=\alpha^{\ell_{E_j}}(x_j)=\bigl(\alpha^j_0\leq \cdots \leq \alpha_{f-1}^j\bigr)$$
the ramification sequence of $\ell_{E_j}$ of $\ell$, where $j=1, \ldots, g$. We denote by $\overline{\alpha}^j:=\alpha^{\ell_{R_1}}(x_j)$
the ramification sequence of the $R_1$-aspect of $\ell$ at $x_j$ for $j=1, \ldots, m$ and by $\overline{\alpha}^j:=\alpha^{\ell_{R_2}}(x_j)$  the ramification sequence of the $R_2$-aspect of $\ell$ at $x_j$ for $j=m+1, \ldots, g$  respectively. Note that 
\begin{equation}\label{eq:compat}
\mathrm{wt}(\alpha^j)+\mathrm{wt}(\overline{\alpha}^j)=f(e-f+1), \  \ \mbox{ for } \ j=1, \ldots, g.
\end{equation}

We  now estimate the dimension of the space of such limit linear series $\ell\in \overline{G}^{f-1}_e(C)$. For $j=m+1, \ldots, g$, the line bundle underlying $\ell_{E_j}$ is equal to $\eta_j(e\cdot x_j)$ and $\ell_j$ is determined by the choice of an $f$-dimensional subspace of sections $V_j\subseteq H^0\bigl(E_j, \eta_j(e\cdot x_j)\bigr)$, satisfying the condition $\alpha^{\ell_{E_j}}(x_j)\geq \alpha^j$. This being a Schubert condition on the Grassmannian $\mathrm{Gr}(f,e)\cong \mathrm{Gr}\bigl(f, H^0(E_j, \eta_j(e\cdot x_j))\bigr)$, it follows that the aspect $\ell_{E_j}$ moves in a parameter space of dimension,  
\begin{equation}\label{eq:ej1}
\mbox{dim } \mbox{Gr}(f,e)-\mathrm{w}(\alpha^j)=f(e-f)-\mathrm{wt}(\alpha^j).
\end{equation}

\vskip 3pt

For $j=1, \ldots, m$, the aspect $\ell_{E_j}$ belongs to the moduli space $\overline{G}^{f-1}_e\bigl(Y_j, (x_j, \alpha^j)\bigr)$. In particular, $\ell_{E_j}$ moves in a moduli space of dimension at most
\begin{equation}\label{eq:ej2}
\rho(1, f-1,e)-\mbox{wt}(\alpha^j)=1+f(e-f)-\mbox{wt}(\alpha^j),    
\end{equation}
see also \cite[Theorem 1.1]{EH}. Furthermore, using once more \cite[Theorem 1.1]{EH}, the $R_1$-aspects of $\ell$ (that is, the collection of the aspect of $\ell$ corresponding to the components of the subcurve $R_1$ of $C$) move in a  space of dimension 
\begin{equation}\label{eq:ej3}
\mbox{dim } \overline{G}^{f-1}_e\Bigl(R_1, (x_1, \overline{\alpha}^1), \ldots, (x_m, \overline{\alpha}^m), (p, \alpha^{\ell_{R_1}}(p))\Bigr)=f(e+1-f)-\sum_{j=1}^m \mathrm{wt}(\overline{\alpha}^j)-\mathrm{wt}\bigl(\alpha^{\ell_{R_1}}(p)\bigr).
\end{equation}
Similarly, the $R_2$-aspect of $\ell$ move in a family of dimension 
\begin{equation}\label{eq:ej4}
\mbox{dim } \overline{G}^{f-1}_e\Bigl(R_2, (x_{m+1}, \overline{\alpha}^{m+1}), \ldots, (x_g, \overline{\alpha}^g), (p, \alpha^{\ell_{R_2}}(p))\Bigr)=f(e+1-f)-\sum_{j=m+1}^g \mathrm{wt}(\overline{\alpha}^j)-\mathrm{wt}\bigl(\alpha^{\ell_{R_2}}(p)\bigr).
\end{equation}
Adding the dimensions in (\ref{eq:ej1}), (\ref{eq:ej2}), (\ref{eq:ej3}) and (\ref{eq:ej4}), and using also (\ref{eq:compat}), as well as the compatibility relation $\mathrm{wt}\bigl(\alpha^{\ell_{R_1}}(p)\bigr)+\mathrm{wt}\bigl(\alpha^{\ell_{R_2}}(p)\bigr)=f(e-f+1)$,  the limit linear series $\ell\in \overline{G}^{f-1}_e(C)$ moves in a family of dimension at most equal to
\begin{align*}
m+mf(e-f)-\sum_{j=1}^m \mathrm{wt}(\alpha^j)+(g-m)f(e-f)-\sum_{j=m+1}^g \mathrm{wt}(\alpha^j)\\
+f(e+1-f)+f(e+1-f)-\sum_{j=1}^{g}\mathrm{wt}(\overline{\alpha}^j)-f(e+1-f)\\
=m-f(g-1-e+f)\leq e-f(g-1-e+f).
\end{align*}
In particular, if $e-f(g-1-e+f)<0$, no such limit linear series $\ell$ can exist which finishes the proof.
\end{proof}

\vskip 5pt

We are now prepared to prove Theorem \ref{thm:genpry}.

\vskip 6pt

\noindent \emph{Proof of Theorem \ref{thm:genpry}}. We assume by contradiction that for every Prym curve $[C, \eta]\in \cR_g$, there exists a component of $V_e^{e-f}(\omega_C\otimes \eta)$ having dimension at least $e-f(g-1-e+f)+1$. 

\vskip 4pt

We degenerate $C$ to a curve $C_0$ of compact type $Y\cup_p E$, where both $[Y,p]\in \cM_{g-1,1}$ and $[E,p]\in \cM_{1,1}$ are general pointed curves of genera $g-1$ and $1$ respectively. We also fix a Prym structure $\eta$ on $C_0$, by choosing \emph{non-trivial} $2$-torsion points $\eta_Y\in \mbox{Pic}^0(Y)[2]$ and $\eta_E\in \mbox{Pic}^0(E)[2]$. By semicontinuity, there exists a family of dimension at least $e-f(g-1-e+f)+1$ of effective divisors $Z$ of degree $e$ on $C_0$, and such that for any line bundle $L_{C_0}\in \mbox{Pic}^e(C_0)$ which is obtained by a twist at $p$ of the line bundle $\eta(Z)$, one has 
$$h^0\bigl(C_0, L_{C_0}\bigr)\geq f.$$
By possibly having to insert chains of rational curves at $p$ and working with the resulting curve instead of $C_0$, we may also assume that the general such $Z$ does not contain $p$ in its support, that is, $\OO_{C_0}(Z)$ is a line bundle.

\vskip 4pt

\noindent {\bf Claim:}  One can find a degeneration $C_0=Y\cup_p E$ as above and such a divisor $Z$ on $C_0$ moving in a family of dimension larger than $e-f(g-1-e+f)$, such that $\mbox{supp}(Z)\subseteq Y$. 

\vskip 4pt

To that end, we invoke once more Proposition 2.2 from \cite{F2}. Precisely, we obtain that there exists a flag curve of genus $g$
$$C=R\cup_{x_1} E_1\cup \ldots \cup_{x_g} E_g$$
consisting of a tree $R$ of smooth rational curves meeting a fixed smooth elliptic curve $E_j$ at the point $x_j$, for $j=1, \ldots, g$, such that all  points of $Z$ specialize on a connected subcurve $Y_1$ of $C$ of arithmetic genus at most $e$ with $\bigl|Y_1\cap \overline{C \setminus Y_1}\bigr|=1$, see also the proof of \cite[Theorem 0.1]{F2}. Since we assumed that $e<g$, denoting by $x_1$ and $x_g$ the extremal points on $R$ among the set $\{x_1, \ldots, x_g\}$, it follows that no point in $\mbox{supp}(Z)$ specializes on at least one of the curves  $E_1$ or $E_g$, for instance on $E_1$. This proves our contention, by smoothing all the nodes of $C$ with the exception of $x_1$. 

\vskip 4pt

Having established the claim, we return to the curve $C_0=Y\cup_p E$ and to the divisor $Z\in Y_e$. We write down the Mayer-Vietoris sequence on $C_0$
$$0\longrightarrow H^0(C_0, \eta(Z))\longrightarrow H^0(Y, \eta_Y(Z))\oplus H^0(E, \eta_E)\longrightarrow \mathbb C_p\longrightarrow\cdots $$
to obtain that $h^0(Y, \eta_Y(Z-p))\geq f$, for a family $Z$ of effective divisors in $Y_e$ of dimension at least $e-f(g-1-e+f)+1$. The divisor $Y_{e-1}\cong p+Y_{e-1}\hookrightarrow Y_e$ being ample, see e.g. \cite[p. 310]{ACGH}, it follows that there exists a non-empty family of effective divisors $Z_1\in Y_{e-1}$ moving in a family of dimension at least $e-f(g-1-e+f)$ such that $h^0(Y, \eta_Y(Z_1))\geq f$, or equivalently by Riemann-Roch 
$$h^0\bigl(Y, \omega_Y\otimes \eta_Y(-Z_1)\bigr)\geq g-2-e+f.$$
We obtain that $\mbox{dim } V_{e-1}^{e-1-f}(\omega_Y\otimes \eta_Y)\geq e-1-f(g-1-e+f)$, that is, the secant locus $V_{e-1}^{e-1-f}(\omega_Y\otimes \eta_Y)$ has excessive dimension. By continuing the descending induction on $g$, we let $Y$ degenerate to a union of a curve of genus $g-2$ and an elliptic tail such that all the secant divisors specialize on the component of genus $g-2$ and at some point we obtain that for a general Prym curve $[Y_1, \eta_1]\in \cR_{g_1}$, for some $g_1\leq g$, satisfies $V_{e_1-1}^{e_1-f}(\omega_{Y_1}\otimes \eta_1)\neq \emptyset$, even though the expected dimension 
$e_1-f(g_1-e_1+f)$ is negative. This on the other hand contradicts Theorem \ref{thm:0nonexistence}, which finishes the proof.
\hfill $\Box$

\section{Prym-canonical secants via Nikulin surfaces}

In this Section we prove both Theorems \ref{thm:disjoint1} and \ref{thm:disjoint2} using Nikulin surfaces.  In order to obtain divisors on $\rr_g$ we focus on the case when the expected dimension of $V_e^{e-f}(\omega_C\otimes \eta)$ equals $-1$, that is, when
\begin{equation}\label{eq:expdim-1}
e-f(g-1-e+f)=-1.
\end{equation} 
In each such case when (\ref{eq:expdim-1}) is satisfied, we define the locus 
\begin{equation}\label{eq:secant-Prym-div}
\cD_g^f:=\Bigl\{[C, \eta]\in \cR_g: V_e^{e-f}(\omega_C\otimes \eta)\neq \emptyset\Bigr\}.
\end{equation}

As explained in the Introduction, we regard the loci $\cD_g^f$ as the Prym analogues of the Brill-Noether subvarieties $\cM_{g,d}^r:=\bigl\{[C]\in \cM_g: W^r_d(C)\neq \emptyset\bigr\}$ of the moduli space of curves. 
\vskip 4pt

When $f=1$, setting $e=i$, we obtain from (\ref{eq:expdim-1}) that $g=2i+1$ and $\cD_{2i+1}^1=\cD_{2i+1}$ is the Prym difference divisor considered in \cite{FL}. From our results, it will follow that $\cD_g^f$ is always an effective divisor on $\cR_g$ whenever (\ref{eq:expdim-1}) is satisfied. When, on the other hand, $e=g$, then via (\ref{eq:expdim-1}) we obtain $g=f^2$ and the condition defining the locus $\cD_g^f$ becomes  
$$h^0(C, A\otimes \eta)\neq 0,$$
for a linear system $A\in W^{f-1}_{f^2-1}(C)$. Note that the Brill-Noether number of such a linear system $A$ equals zero, therefore in this case $\cD_g^f$ can be regarded as a (global version of a) translate of a multi-theta theta divisor on the moduli space $\cR_g$.

\vskip 5pt

We fix a standard Nikulin surface $X$ of genus $g$ such that 
$\mbox{Pic}(X)\cong \mathfrak{N}\oplus_{\perp} \mathbb Z\cdot L$, where $L^2=2g-2$. Recall from the Introduction that $N_1, \ldots, N_8$ are the disjoint $(-2)$-curves on $X$ and that $L\cdot N_j=0$, for $j=1, \ldots, 8$. We have the following slight strengthening of Theorem \ref{thm:disjoint1} from the Introduction.

\begin{theorem}\label{thm:non-existence}
For every standard Nikulin surface $X$ of genus $g\geq 6$ with $\mathrm{Pic}(X)\cong \mathfrak{N}\oplus_{\perp} \mathbb Z\cdot L$ and for integers $0\leq f\leq e$. If $e-f(r+1-e+f)<0$, then  $V_e^{e-f}(\omega_C\otimes \mathfrak{e}_C)=\emptyset$, for every smooth curve $C\in |L|$. 
\end{theorem}

\begin{proof}
We fix a smooth curve $C\in |L|$ and set $\eta:=\mathfrak{e}_C^{\vee}$. Using \cite[p. 356]{ACGH},  since $\omega_C\otimes \eta$ is a non-special linear system, we obtain that $V_e^{e-f}(\omega_C\otimes \eta)\neq \emptyset$, whenever $e-f(g-1-e+f)\geq 0$. Therefore it suffices  to show that $V_e^{e-f}(\omega_C\otimes \eta)=\emptyset$, when $\mbox{exp.dim } V_e^{e-f}(\omega_C\otimes \eta)<0$.

\vskip 4pt

Suppose, by contradiction, that $Z=x_1+\cdots+x_e$ is an element of $V_e^{e-f}(\omega_C\otimes \eta)$. Note that we do not require the points $x_j$ to be distinct. By definition, $Z$ may be regarded as a length $e$ curvilinear subscheme of $X$. Then $h^0(C, \omega_C\otimes \eta(-Z))\geq h^0(C, \omega_C\otimes \eta)-e+f=g-1-e+f$. 

\vskip 4pt

Observe first that $H^i(X, \mathfrak{e})=0$ for $i=0,1,2$. By twisting the exact sequence $$0\longrightarrow \OO_X(-C) \longrightarrow \I_{Z/X}\longrightarrow \OO_C(-Z)\longrightarrow 0$$ by $L\otimes \mathfrak{e}^{\vee}$ and then taking cohomology, we obtain the exact sequence
$$0\longrightarrow H^0(X, \mathfrak{e}^{\vee})\longrightarrow H^0\bigl(X, \I_{Z/X}\otimes L\otimes \mathfrak{e}^{\vee}\bigr) \longrightarrow H^0\bigl(C, \omega_C\otimes \mathfrak{e}_C(-Z)\bigr)\longrightarrow H^1(X, \mathfrak{e}^{\vee})\longrightarrow \cdots,$$
which implies that
\begin{equation}\label{eq:isom3}
H^0\bigl(X, \I_{Z/X}\otimes L\otimes \mathfrak{e}^{\vee}\bigr)\cong H^0\bigl(C, \omega_C\otimes \eta(-Z)\bigr).
\end{equation}
We consider the class $H:=L\otimes \mathfrak{e}^{\vee}\in \mbox{Pic}(X)$. Note that $H^2=2g-6$, therefore curves $D\in |H|$ have arithmetic genus $g-2$. If $g\geq 6$, using \cite[Lemma 3.1]{GS} we know that the linear system $|H|$ is very ample. In particular, using (\ref{eq:isom3}), we can choose a smooth curve $D\in |H|$ passing through the points $x_1, \ldots, x_e$.

\vskip 4pt

We tensor the exact sequence $0\rightarrow \OO_X(-D)\rightarrow \I_{Z/X}\rightarrow \OO_D(-Z)\rightarrow 0$ by the line bundle $H$ and taking cohomology, we obtain the exact sequence
$$0\longrightarrow H^0(X, \OO_X)\longrightarrow H^0\bigl(X, \I_{Z/X}\otimes L\otimes \mathfrak{e}^{\vee}\bigr) \longrightarrow H^0\bigl(D, \omega_D(-Z)\bigr)\longrightarrow H^1(X, \OO_X)\longrightarrow \cdots,$$
from which we compute via the isomorphism (\ref{eq:isom3}) that $h^0(D, \omega_D(-Z))\geq g-2-e+f$.
Therefore, we can write
$$h^0\bigl(D, \OO_D(Z)\bigr)=h^0\bigl(D, \omega_D(-Z)\bigr)+e+1-g+2 \geq g-2-e+f+e-g+3=f+1,$$
that is $\OO_D(Z)\in W_e^f(D)$. Computing the Brill-Noether number for $\OO_D(Z)$, we observe 
$$\rho(g-2,f,e)=g-2-(f+1)(g-2-e+f)=e-f(g-1-e+f)=\mbox{ exp.dim } V_e^{e-f}(\omega_C\otimes \eta)<0,$$
that is, the curve $D$ is Brill-Noether special, precisely $W^f_e(D)\neq \emptyset$.

\vskip 4pt

We show that this is not possible. Indeed, using the arguments in Lazarsfeld's proof \cite[Lemma 1.3]{La1} of Petri's theorem for curves on surfaces $K3$ (see also \cite{La2}), to rule out this possibility, it suffices to show that there is no decomposition $L\otimes \mathfrak{e}^{\vee}=A_1+A_2\in \mbox{Pic}(X)$, where $h^0(X,A_i)\geq 2$ for $i=1,2$.

\vskip 4pt

Suppose there is such a decomposition and we write $A_1 \equiv a_1 L+b_1 N_1+\cdots +b_8 N_8$, respectively $A_2=a_2 L+c_1 N_1+\cdots +c_8 N_8$. Then $a_1$ and $a_2$ must be non-negative and since $a_1+a_2=1$, without loss of generality we may assume $a_1=0$, that is, $A_1\in \mathfrak{N}$. But then $|A_1|$ can consists only of sums of the rational curves $N_1, \ldots, N_8$, that is, $h^0(X,A_1)\leq 1$, which yields the desired contradiction. Therefore $V_e^{e-f}(\omega_C\otimes \eta)=\emptyset$, when $e-f(g-1-e+f)<0$.

\vskip 5pt

\end{proof}

\begin{remark}\label{rmk:1_nodal}
The conclusion of Theorem \ref{thm:non-existence} can be extended to any irreducible $1$-nodal curve $C\in |L|$, while keeping the same assumptions. Indeed, suppose $C$ is such a $1$-nodal curve with $\mbox{Sing}(C)=\{o\}$, and denote by $\nu\colon C'\rightarrow C$ the normalization map with $x,y\in C'$ such that $\nu(x)=\nu(y)=o$. Assuming $[C, \mathfrak{e}_C]$ is a limit of smooth Prym curves $[C_t, \eta_t]\in \cR_g$, with $V_e^{e-f}(\omega_{C_t}\otimes \eta_t)\neq \emptyset$, then if $Z$ is the limit on $C$ of the $e$-secant divisors on $C_t$, in the case $o \notin \mbox{supp}(Z)$ the proof remains unchanged.  Assume for simplicity that \emph{one} of the points in $\mbox{supp}(Z)$, say $x_1$, specializes to $o$, while $x_2, \ldots, x_e\in C_{\mathrm{reg}}$. (The case when several points in $\mbox{supp}(Z)$ specialize to $o$ is similar). We denote by $\epsilon\colon X'\rightarrow X$ the blow-up of $X$ at $o$, by $E$ the exceptional divisor and by $C'\subseteq X'$ the proper transform of $C$. Set $Y=\epsilon^*(C)=C'+E$.  Then  $C'\cap E=\{x,y\}$. We denote by $Z'=x_1+x_2+\cdots+x_e$ the limiting $e$-secant divisor, where $x_1\in E\setminus \{x,y\}$ and $x_2, \ldots, x_e\in C'\setminus \{x,y\}$. We then have 
$$H^0\bigl(X', \I_{Z'/X'}\otimes \epsilon^*(L\otimes \mathfrak{e}^{\vee})\bigr) \cong H^0\bigl(Y, \epsilon^*(L\otimes \mathfrak{e}^{\vee})(-Z')\bigr),$$
and pushing down to $X$, we obtain again that $h^0\bigl(X,\I_{Z/X}\otimes (L\otimes \mathfrak{e}^{\vee})\bigr)\geq g-1-e+f$. The rest of the proof is identical to that of Theorem \ref{thm:non-existence}.
    
\end{remark}

\begin{remark}\label{rmk:ns3}
It is natural to ask whether the first part of Theorem \ref{thm:non-existence} can be extended to reducible curves $C\in|L|$ containing one of the $N_i$, without loss of generality $N_1$. Thus we assume that $e-f(g-1-e+f)<0$ and  consider a curve $C=C'+N_1$, where $C'\in |L-N_1|$ is smooth and irreducible and denote by $x,y$ the two intersection points between $C'$ and $N_1$. The restrictions of $\eta$ to the two components of $C$ are given by $\eta_{C'}\cong \mathfrak{e}_{C'}^\vee$,  respectively $\eta_{N_1}\cong  \OO_{\PP^1}(1)$. 


Assume $[C, \eta]\in \dd_g^f$, which implies that there exists a divisor $Z\in V_e^{e-f}(\omega_C\otimes \eta)$. We denote by $Z'$ and $Z_1$ the restriction of $Z$ to $C'$ and $N_1$ respectively, and set $e':=\deg Z'$ and $e_1:=\deg Z_1$. Without loss of generality we may assume that $Z$ is disjoint from $x$ and $y$, so that $e=e'+e_1$. We distinguish two cases according to the parity of $e_1$.

\vskip 4pt

Assume first $e_1$ is even and write $e_1=2a$. By considering the twist by $\eta^{\otimes 2a}$ of the Prym-canonical line bundle, by semicontinuity we also have $Z\in V_e^{e-f}(\omega_C\otimes \eta^{\otimes (2a+1)})$. The restrictions of the line bundle $\omega_C\otimes \eta^{\otimes (2a+1)}(-Z)$ to the components $N_1$ respectively $C'$ are isomorphic to $\OO_{N_1}(1)$ respectively 
$\omega_{C'}\otimes \mathfrak{e}_{C'}^{\vee}(-Z_2)$, where $Z_2=Z'+(a-1)\cdot x+(a-1)\cdot y)$. Note that $\mbox{deg}(Z_2)=e-2$.

\vskip 4pt

The condition  $h^0\bigl(C, \omega_C\otimes \eta^{\otimes (2a+1)}(-Z)\bigr)\geq g-1-e+f$ translates then into 
\begin{equation}\label{eq:ram4}
h^0\bigl(C', \omega_{C'}\otimes \mathfrak{e}_{C'}^{\vee}(-Z_2)\bigr)\geq g-1-e+f\Longleftrightarrow h^0\bigl(X, \I_{Z_2/X}(C'\otimes \mathfrak{e}^{\vee})\bigr)\geq g-1-e+f.
\end{equation}

We take a general element $D$ of the linear system $\bigl|\I_{Z_2/X}(C'\otimes \mathfrak{e}^{\vee})\bigr|$. Observe that $D$ is a smooth curve of genus $g-4$ and proceeding precisely along  the lines of the proof of Theorem \ref{thm:non-existence}, we conclude from (\ref{eq:ram4}) that $h^0\bigl(D, \OO_D(Z_2)\bigr)\geq f+1$, that is, $W_{e-2}^f(D)\neq \emptyset$.
Observe now that the Brill-Noether number for such linear systems equals 
$$\rho(g-4,e-2,f)=e-f(g-1-e+f)-2=\mbox{exp.dim } V_e^{e-f}(\omega_C\otimes \eta) \leq -3,$$
that is, $D$ is a Brill-Noether special curve. This is ruled out like in the proof of Theorem \ref{thm:non-existence}.

\vskip 4pt

Assume now $e_1$ is odd and write $e_1=2a+1$. Then, also  $Z\in V_e^{e-f}(\omega_C\otimes \eta^{\otimes (2a+1)})$. The restrictions of $\omega_C\otimes \eta^{\otimes (2a+1)}(-Z)$ to $N_1$ respectively to $C'$ are isomorphic to $\OO_{N_1}$ respectively $\omega_{C'}\otimes \mathfrak{e}_{C'}^{\vee}\bigl(-
(Z'+(a-1)x+(a-1)y)\bigr)$. Setting $Z_3:=Z'+a\cdot x+a\cdot y$, we obtain that 
\begin{equation}\label{eq:ram5}
H^0\bigl(C', \omega_{C'}\otimes \mathfrak{e}_{C'}^{\vee}(-Z_3)\bigr)\geq g-e+f-2\Longleftrightarrow h^0\bigl(X, \I_{Z_3/X}(C'-\mathfrak{e})\bigr)\geq g-2-e+f.
\end{equation}

\vskip 4pt

Choosing  a general curve $D\in \bigl|\I_{Z_3/X}(C'\otimes \mathfrak{e}^{\vee})\bigr|$, we obtain via the equivlence (\ref{eq:ram5})  that $\OO_D(Z_3)\in W^f_{e-1}(D)$. 
We now compute the Brill-Noether number
$$\rho(g-4,f,e-1)=f-1+e-f(g-1+e+f)=f-1+\mbox{exp.dim } V_e^{e-f}(\omega_C\otimes \eta).$$
We conclude like at the end of the proof of Theorem \ref{thm:non-existence} that this situation does not appear, at least as long as $f=1$. We summarize that when $f=1$, the  curves $C'+N_1$ do not lie in $\dd_g^1$.

\end{remark}

\section{The Prym difference divisor via Raynaud theta divisors}

In this Section we present an alternative approach to both Theorem \ref{thm:disjoint1} (in the case $f=1$) and Theorem \ref{thm:disjoint2} (in full generality) using the representation of difference varieties on Jacobians as Raynaud theta divisors provided in \cite{FMP} in the context of the resolution of the Minimal Resolution Conjecture for points on canonical curves.

\vskip 5pt

We begin by setting some notation. For a smooth curve $C$ and for integers $a, b>0$, we denote by $C_a-C_b\subseteq \mbox{Pic}^{a-b}(C)$ its \emph{difference variety} defined as the image of the difference map 
$$\phi_{a,b}\colon C_a\times C_b\longrightarrow \mbox{Pic}^{a-b}(C), \ \ (D_a, D_b)\mapsto \OO_C(D_a-D_b).$$
Of particular interest is the case when $a=b=i$ and $g=2i+1$, in which case $C_i-C_i\subseteq \mbox{Pic}^0(C)$ is a divisor. Using \cite{FMP} this divisor can be identified with the \emph{Raynaud theta divisor} \cite{R} associated to the kernel (syzygy) vector bundle on $C$. Precisely, we let $M_{\omega_C}$ be the \emph{kernel bundle} defined by the exact sequence
\begin{equation}\label{eq:M-bundle}
0\longrightarrow M_{\omega_C}\longrightarrow H^0(C, \omega_C)\otimes \OO_C\longrightarrow \omega_C\longrightarrow 0.
\end{equation}
Note that the exact sequence (\ref{eq:M-bundle}) makes sense for every nodal curve $C$ for which $\omega_C$ is globally generated. We set $Q_{\omega_C}:=M_{\omega_C}^{\vee}$ and observe that the slope of $Q_{\omega_C}$ equals $\mu\bigl(Q_{\omega_C}\bigr)=2$. Therefore $\mu\bigl(\bigwedge^i Q_{\omega_C}\bigr)=2i=g-1$. The main results of \cite{FMP} establishes the following equality of cycles on $\mbox{Pic}^0(C)$:
\begin{equation}\label{eq:FMP}
C_i-C_i=\Theta_{\bigwedge^i Q_{\omega_C}}:=\Bigl\{\xi\in \mbox{Pic}^0(C): h^0\bigl(C, \bigwedge^i Q_{\omega_C}\otimes \xi\bigr)\geq 1\Bigr\}.
\end{equation}
It is advantageous to observe that, whereas the left-hand side of (\ref{eq:FMP}) is hard to understand for singular stable curves, the definition of the right-hand side can be extended to cover those stable curves $C$ for which $\omega_C$ is globally generated.

\vskip 4pt

We are now prepared to provide a second proof of the non-existence part of Theorem \ref{thm:disjoint1} in the case $f=1$. 

\begin{theorem}\label{thm:nikulin_disjoint1}
Let $[X, \ \mathfrak{N}\oplus_{\perp} \mathbb Z\cdot  L]\in \F_{2i+1}^{\mathfrak{N}}$ be a general standard Nikulin surface of genus $g=2i+1$. Then for every integral nodal curve $C\in |L|$, we have that $[C, \mathfrak{e}_C^{\vee}]\notin \overline{\cD}_{2i+1}$.
\end{theorem}
\begin{proof}
Using \cite{vGS} we may assume that in this case $\mbox{Pic}(X)\cong \mathfrak{N}\oplus_{\perp} \mathbb Z\cdot L$. The line bundle $L$ is nef and curves $C\in |L|$ have genus $g$ and $C\cdot N_j=0$, for $j=1, \ldots, 8$.

\vskip 3pt

We choose a nodal integral curve $C\in |L|$ and recall that $Q_{\omega_C}=M_{\omega_C}^{\vee}$.  By requiring the $2$-torsion line bundle of a stable Prym curve to lie in the left hand side of $(\ref{eq:FMP})$, one obtains a codimension one subvariety of $\rr_{2i+1}$ that contains the closure $\dd_{2i+1}$ of $\cD_{2i+1}$, as well as possibly other boundary divisors of $\rr_{2i+1}$. Throughout this proof we set $\eta:=\mathfrak{e}_{C}\in \mbox{Pic}^0(C)[2]$. In order to conclude that $[C, \eta]\notin \dd_{2i+2}$, it suffices thus to show that 
\begin{equation}\label{eq:suff1}
H^0\bigl(C, \bigwedge^i Q_{\omega_C} \otimes \mathfrak{e}_C\bigr)=0.
\end{equation}
We now specialize $X$ to a \emph{hyperelliptic} Nikulin surface $X'$, that is, a $K3$ surface with  
$$\mbox{Pic}(X')\cong \mathfrak{N}\oplus_{\perp} \mathbb Z\cdot L\oplus \mathbb Z\cdot E,$$
where $E^2=0$, $E\cdot L=2$ and $E\cdot N_i=0$, for $i=1, \ldots, 8$. By the Torelli theorem for $K3$ surfaces, see e.g. \cite{Dol}, there exists a $K3$ surface $X'$ having this Picard lattice. Moreover, 
$[X', \mathfrak{N}\oplus_{\perp} \mathbb Z\cdot L]\in \F_g^{\mathfrak{N}}$, that is, $X'$ is a standard Nikulin surface of genus $g$. We may assume that both $C$ and $E$ are nef classes on $X'$. Since $|E|$ is an elliptic pencil and $C\cdot E=2$, it follows that any $C\in |L|$ is hyperelliptic. Furthermore, $\omega_C=E_C^{\otimes (g-1)}$  and from (\ref{eq:M-bundle}) we obtain that 
$M_{\omega_C}\cong \bigl(E_C^{\vee}\bigr)^{\oplus (g-1)}$ {\bf(cf. \cite[Proposition 3.5]{FMP})}, therefore $Q_{\omega_C} \cong E_C^{\oplus (g-1)}$. It follows that 
$$\bigwedge^i Q_{\omega_C}\cong \Bigl(E_C^{\otimes i}\Bigr)^{\oplus {g-1\choose i}},$$
therefore condition (\ref{eq:suff1}), amounts to 
\begin{equation}\label{eq:suff2}
H^0\bigl(C, E_C^{\otimes i}\otimes \mathfrak{e}_C\bigr)=0.
\end{equation}
By tensoring the exact sequence $0\rightarrow \OO_{X'}(-C)\rightarrow \OO_{X'}\rightarrow \OO_C\rightarrow 0$ by  $\OO_{X'}(iE-\mathfrak{e})$ and taking cohomology, then (\ref{eq:suff2}) follows once we show that $$H^0\bigl(X', \OO_{X'}(iE-\mathfrak{e})\bigr)=0 \  \mbox{ and  } \ H^1\bigl(X', \OO_{X'}(iE-\mathfrak{e}-C)\bigr)=0.$$ 

The first statement follows from Lemma 5.8 in \cite{FK} (note that the assumption $g\geq 11$ in \emph{loc.cit.} is not used anywhere in the proof). For the second statement, since $(iE-\mathfrak{e}-C)^2=-4$, while clearly $H^0\bigl(X', \OO_{X'}(iE-\mathfrak{e}-C)\bigr)=0$ (intersect with the nef class $E$), it follows by Riemann-Roch coupled with Serre duality that 
$$h^1\bigl(X', \OO_{X'}(iE-C-\mathfrak{e})\bigr)=h^2\bigl(X', \OO_{X'}(iE-C-\mathfrak{e})\bigr)=h^0\bigl(X', \OO_{X'}(C+\mathfrak{e}-iE)\bigr).$$ Thus we are left with proving that $H^0\bigl(X', \OO_{X'}(C+\mathfrak{e}-iE)\bigr)=0$. Assuming by contradiction that $C+\mathfrak{e}-iE$ is effective, by intersecting with $N_i$, since $N_i\cdot (C+\mathfrak{e}-iE)=-1$, we obtain that also $C+\mathfrak{e}-iE-N_1-\cdots-N_8=C-\mathfrak{e}-iE$ is effective. Below we shall show that this is not possible, thus completing the proof.
\end{proof}

\begin{lemma}\label{lemma:vanhyp_1}
Let $X'$ be a general standard hyperelliptic Nikulin surface of genus $2i+1$ as above. Then 
$$H^0\bigl(X', \OO_{X'}(C-iE-\mathfrak{e})\bigr)=0.$$
\end{lemma}

\begin{proof}
We set $B\equiv L-iE-\mathfrak{e}\in \mbox{Pic}(X')$. Since $B^2=(C-iE)^2+\mathfrak{e}^2=-4$, if $B$ is effective, there must exist a $(-2)$-curve $D\subseteq X'$ such that $B-D$ is effective and $D\cdot B<0$. We write 
\begin{equation}\label{eq:D1}
D\equiv aL+bE+c_1 N_1+\cdots+c_8 N_8\in \mbox{Pic}(X'),
\end{equation}
where $a, b\in \mathbb Z$ and $c_j\in \frac{1}{2}\mathbb Z/\mathbb Z$.

\vskip 4pt
Since $B\cdot N_j=1$, it follows that $D\neq N_j$, therefore $D\cdot N_j\geq 0$, that is, $c_j\leq 0$, for $j=1, \ldots, 8$.
Intersecting (\ref{eq:D1}) with the nef class $E$, we obtain $a\geq 0$. On the other hand, also $(B-D)\cdot E\geq 0$, which yields $a\leq 1$, that is, $a\in \{0,1\}$, since from (\ref{eq:D1}) we have $a\in \mathbb Z$.

\vskip 3pt 

Assume first $a=1$. From the inequality $(B-D)\cdot C\geq 0$, we obtain  $i+b\leq 0$. On the other hand, then $D^2= (C+bE)^2-2(c_1^2+\cdots+c_8^2)\leq 4(i+b)$, which forces $b=-i$, since  $b\in \mathbb Z$. Since $D^2=-2$, from (\ref{eq:D1}) we obtain that $c_1^2+\cdots+c_8^2=1$.
Observe that since $c_1N_1+\cdots+c_8N_8\in \mathfrak{N}$, the integers $2c_1, \ldots, 2c_8$ are all of the same parity. Therefore we conclude that there exists $j\in \{1, \ldots, 8\}$ such that $c_j=-1$, while $c_{\ell}=0$, for $\ell \neq 8$. We obtain that the class $B-D=N_j-e$ is effective, which is obviously impossible.

\vskip 3pt

Assume now $a=0$. Since $D\cdot C\geq 0$, we obtain $b\geq 0$. Since $D^2=-2$, we again obtain $c_1^2+\cdots+c_8^2=1$, and the same reasoning as above implies that there exists $j\in \{1, \ldots, 8\}$ such that $c_j=-1$, while $c_{\ell}=0$, for $\ell \neq j$. We then have $0>B\cdot D=2b+1\geq 0$, therefore, we have reached a contradiction. 
\end{proof}

\subsection{Prym curves on non-standard Nikulin surfaces}

We start with a general non-standard Nikulin surface of genus $h=8i-1$, where $i\geq 2$. Then using \cite[Proposition 2.1]{vGS}, or \cite{KLV1}, there exist disjoint $(-2)$-curves $N_1, N_2\in \mathfrak{N}\subseteq \mbox{Pic}(X)$, such that if
$\mathfrak{e}\cong \frac{N_1+\cdots+N_8}{2}\in \mbox{Pic}(X)$, where $\mathfrak{N}=\langle N_1, \ldots, N_8, \mathfrak{e}\rangle$, then 
$$R\equiv \frac{L-N_1-N_2}{2}\in \mbox{Pic}(X).$$
We denote by $N_3, \ldots, N_8$ the remaining effective $(-2)$-curves in $\mbox{Pic}(X)$.

\vskip 4pt

Then $R^2=4i-2$, that is, curves $C\in |R|$ have genus $2i$. Note that $R\cdot N_1= R\cdot N_2=1$, while $L\cdot N_i=0$, for $i=3, \ldots, 8$. We fix an irreducible such curve $C$ and assume it to be smooth at the points $C\cdot N_1$ and $C\cdot N_2$. Retaining the notation from the Introduction, if $C\cdot N_i=\{x_i\}$, for $i=1,2$ and $f\colon \widetilde{X}\rightarrow X$ is the double cover induced by $e$, then $\mathfrak{e}_C^{2}\cong \OO_C(x_1+x_2)$, that is, the double cover $f_{|f^{-1}(C)}\colon f^{-1}(C)\rightarrow C$ is branched only over the points $x_1$ and $x_2$. We set 
$$Y:=C\cup_{\{x_1, x_2\}} \PP^1,$$ where $\PP^1$ meets $C$ transversally at $x_1$ and $x_2$ and we let $\eta$ be the line bundle on $Y$ having restrictions $\eta_{\PP^1}\cong \OO_{\PP^1}(1)$ and $\eta_{C}\cong \mathfrak{e}_C^{\vee}$. Then $\xi(C):=[Y, \eta]\in \Delta_0^{\mathrm{ram}}$ can be regarded as a stable Prym curve of genus $2i+1$. 

\vskip 3pt

We are now in a position to complete the proof of Theorem \ref{thm:disjoint2}

\vskip 5pt

\noindent \emph{Proof of Theorem \ref{thm:disjoint2}.} Retaining the notation  above,  in order to conclude that $[Y, \eta]\notin \dd_{2i+1}$, it suffices to show that 
\begin{equation}\label{eq:suff22}
H^0\bigl(Y,\bigwedge^i Q_{\omega_Y}\otimes \eta^{\vee}\bigr)=0,\,\,\,\,\textrm{or equivalently,}\,\,\,\, H^0\bigl(Y, \bigwedge^i M_{\omega_Y}\otimes \omega_Y\otimes \eta\bigr)=0.
\end{equation}
where the equivalence of the above vanishing statements follow from the fact that $Q_{\omega_Y}$ is a vector bundle of slope 
$\mu\bigl(Q_{\omega_Y}\bigr)=2$.

\vskip 3pt 

We have $M_{\omega_Y|C}\cong M_{\omega_C(x_1+x_2)}$ and
$M_{\omega_Y|\PP^1}\cong \OO_{\PP^1}^{\oplus 2i}$. One has the following Mayer-Vietoris type exact sequence on $Y$
$$
0\longrightarrow \bigwedge^i M_Y\longrightarrow \bigwedge^i M_{\omega_C(x_1+x_2)}\bigoplus \OO_{\PP^1}^{\oplus {2i\choose i}} \longrightarrow \bigwedge^i M_{\omega_Y|x_1+x_2}\longrightarrow 0.
$$
We twist this sequence by the line bundle $\omega_Y\otimes \eta$. Observing that 
$\bigl(\omega_Y\otimes \eta\bigr)_{|\PP^1}\cong \OO_{\PP^1}(1)$ and that $\bigl(\omega_Y\otimes \eta\bigr)_{|C}\cong \omega_C\otimes \mathfrak{e}_C$, as well as the existence of the isomorphism 
$$H^0\bigl(\PP^1, \OO_{\PP^1}(1)\bigr)^{\oplus {2i\choose i}}\stackrel{\cong}\longrightarrow \bigwedge^i M_{\omega_Y}\otimes (\omega_Y\otimes\eta)_{|x_1+x_2},$$
it follows that the statement (\ref{eq:suff22}) is equivalent to the following vanishing on $C$
\begin{equation}\label{eq:suff3}
H^0\bigl(C, \bigwedge^i M_{\omega_C(x_1+x_2)}\otimes \omega_C\otimes \mathfrak{e}_C\bigr)=0.
\end{equation}

We then have the following exact sequence on $C$ 
$$0\longrightarrow M_{\omega_C}\longrightarrow M_{\omega_C(x_1+x_2)}\longrightarrow \OO_C(-x_1-x_2)\longrightarrow 0,$$ which after taking exterior products leads to the following exact sequence
$$0\longrightarrow \bigwedge ^i M_{\omega_C}\longrightarrow \bigwedge ^i M_{\omega_C(x_1+x_2)}\longrightarrow \bigwedge^{i-1} M_{\omega_C}(-x_1-x_2)\longrightarrow 0.$$

Twisting this sequence by $\omega_C\otimes \mathfrak{e}_C$ and taking cohomology, we conclude that the vanishing (\ref{eq:suff3}) holds if  one has both  
\begin{equation}\label{eq:suff4}
H^0\bigl(C, \bigwedge^i M_{\omega_C}\otimes \omega_C\otimes \mathfrak{e}_C\bigr)=0  \ \  \mbox{ and } \ \ H^0\bigl(C, \bigwedge^{i-1} M_{\omega_C}\otimes \omega_C\otimes \mathfrak{e}_C^{\vee}\bigr)=0.
\end{equation}

\vskip 3pt

We specialize $X$ to a non-standard Nikulin surface $X'$ having the Picard lattice
$$\mbox{Pic}(X')\cong\mathfrak{N}\oplus \mathbb Z\cdot R\oplus \mathbb Z\cdot E,$$
where $E^2=0$, $E\cdot L=4$ and $E\cdot N_j=0$, for $j=1, \ldots, 8$. By \cite{Dol} it follows that a Nikulin surface $X'$ with this Picard lattice exists, furthermore both $R$ and $E$ are nef classes. Observe that $E\cdot R=2$, that is,  a curve $C\in |R|$ is hyperelliptic. Accordingly,  $M_{\omega_C}\cong \bigl(E_C^{\vee}\bigr)^{\oplus (2i-1)}$ and since $\omega_C\cong E_C^{\otimes (2i-1)}$, we conclude that (\ref{eq:suff4}) is equivalent to the following vanishing statements 
\begin{equation}\label{eq:suff5}
H^0\bigl(C, E_C^{\otimes (i-1)}\otimes \mathfrak{e}_C\bigr)=0 \ \ \mbox{ and } \ \ H^0\bigl(C, E_C^{\otimes i}\otimes \mathfrak{e}_C^{\vee}\bigr)=0.
\end{equation}
Twisting the exact sequence $0\rightarrow \OO_{X'}(-C)\rightarrow \OO_{X'}\rightarrow \OO_C\rightarrow 0$ by $\OO_{X'}(iE-\mathfrak{e})$ respectively by $\OO_{X'}((i-1)E+\mathfrak{e})$, we conclude that (\ref{eq:suff5}) are implied by the following vanishing statements:
\begin{equation}\label{eq:suff6}
H^0\bigl(X', \OO_{X'}((i-1)E+\mathfrak{e})\bigr)=0 \ \ \mbox{ and } \ \ H^0\bigl(X', \OO_{X'}(C-(i-1)E-\mathfrak{e})\bigr)=0,
\end{equation}
respectively
\begin{equation}\label{eq:suff7}
H^0\bigl(X', \OO_{X'}(iE-\mathfrak{e})\bigr)=0 \ \ \mbox{ and } \ \ H^0\bigl(X', \OO_{X'}(C-iE+\mathfrak{e})\bigr)=0.
\end{equation}
Statements (\ref{eq:suff6}) and (\ref{eq:suff7}) are established in the next lemma, bringing the proof to an end.
\hfill $\Box$

\begin{lemma}\label{ref:van_nik3}
Let $X'$ be a non-standard hyperelliptic Nikulin surface with $\mathrm{Pic}(X')\cong \mathbb Z\cdot R\oplus \mathfrak{N}\oplus \mathbb Z\cdot E$, where $R^2=4i-2$, $R\cdot E=2$ and $E\perp \mathfrak{N}$ as above. Then the following hold:
\begin{align*}
H^0\bigl(X', \OO_{X'}((i-1)E+\mathfrak{e})\bigr)=0 \ \ \mbox{ and } \ \ H^0\bigl(X', \OO_{X'}(iE-\mathfrak{e})\bigr)=0,\\
H^0\bigl(X', \OO_{X'}(R-(i-1)E-\mathfrak{e})\bigr)=0 \ \ \mbox{ and } \ \ H^0\bigl(X', \OO_{X'}(R-iE+\mathfrak{e})\bigr)=0.
\end{align*}
\end{lemma}
\begin{proof}
 Observe that each of the divisor classes appearing in the statement above has self-intersection $(-4)$ and the proofs are quite similar. We will provide details only for the last statement, that is, $H^0\bigl(X', \OO_{X'}(R-iE+\mathfrak{e})\bigr)=0$. 

 \vskip 4pt

 We proceed by contradiction and assume $H^0\bigl(X', \OO_{X'}(R-iE+\mathfrak{e})\bigr)\neq 0$. By intersecting with each of the curves $N_3, \ldots, N_8$ and recalling that $R\equiv 
 \frac{L-N_1-N_2}{2}$, we obtain that
$$H^0\Bigl(X', \OO_{X'}\Bigl(\frac{L}{2}-iE-\frac{N_3}{2}-\cdots-\frac{N_8}{2}\Bigr)\Bigr)\neq 0.$$
We set $$B:=\frac{L}{2}-iE- \frac{N_3}{2}-\cdots -\frac{N_8}{2}\in \mbox{Pic}(X').$$
Since $B^2= -4$, just like in the proof of Lemma \ref{lemma:vanhyp_1}, there exists an effective $(-2)$-curve 
\begin{equation}\label{eq:D2}
D\equiv aL+bE+c_1N_1+\cdots+c_8N_8\in \mbox{Pic}(X'),
\end{equation}
such that $B-D$ is effective and $B\cdot D<0$.

\vskip 4pt

Since $B\cdot N_j\geq 0$, it follows that $D\cdot N_j\geq 0$, therefore $c_j\leq 0$, for $j=1, \ldots, 8$. Furthermore, $a\geq 0$, since $D\cdot E\geq 0$. The class $E$ being nef and $B-D$ being effective, one has $(B-D)\cdot E\geq 0$, which yields $a\leq \frac{1}{2}$, that is, $a\in \{0, \frac{1}{2}\}$, since from (\ref{eq:D2}) it follows that $a\in \frac{1}{2}\mathbb Z/\mathbb Z$.

\vskip 4pt

Assume first $a=\frac{1}{2}$. From the inequality $(B-D)\cdot L\geq 0$, we obtain $i+b\leq 0$. Furthermore, we have 
$-2=D^2=4(b+i)-1-2(c_1^2+\cdots+c_8^2)$. Since $b\in \mathbb Z$, it follows $b+i=0$ and $c_1^2+\cdots+c_8^2=\frac{1}{2}$. The only solution of this equality compatible with the description (\ref{eq:D2}), is when $c_1=c_2=-\frac{1}{2}$ and $c_3=\cdots=c_8=0$. We obtain that  $\frac{1}{2}\bigl(N_1+N_2-N_3-\cdots-N_8\bigr)\in \mbox{Pic}(X')$ is effective, which is impossible.

\vskip 4pt

We are left with the case $a=0$. Since $D\cdot L\geq 0$, we obtain $b\geq 0$. From $D^2=-2$, we find $c_1^2+\cdots+c_8^2=1$. Note that the integers $2c_1$ and $2c_2$, respectively $2c_3, \ldots, 2c_8$ must have the same parity, which leaves as only cases possible $c_1=-1$ and $c_2=0$ (respectively $c_1=0$ and $c_2=-1$) and $c_3=\ldots=c_8=0$. In both cases we then compute that $D\cdot B=2b\geq 0$, which yields the desired contradiction.
\end{proof}

\section{The class of $\dd_{2i+1}$}

We now explain how the geometric information contained in Theorem \ref{thm:disjoint1} and \ref{thm:disjoint2} suffices to determine the class $[\dd_{2i+2}]$ at least in the $\bigl\langle \lambda, \delta_0', \delta_0^{\mathrm{ram}}\bigr\rangle$-subspace of $\mbox{Pic}(\rr_{2i+1})$. As explained in \cite{FL}, it is precisely these three coefficients of $[\dd_{2i+1}]$ that are relevant for Kodaira dimension calculations of $\rr_g$.

\vskip 3pt

Recall that if $\pi\colon \rr_g\rightarrow \mm_g$ is the finite morphism forgetting the Prym structure of every Prym curve, then we have a relation 
$$\pi^*(\delta_0)=\delta_0'+\delta_0''+2\delta_0^{\mathrm{ram}}\in \mbox{Pic}(\rr_{g}).$$

\vskip 3pt

We briefly describe the modular meaning of these divisor classes. If we  fix a point $[C_{xy}]\in \Delta_0$ induced by a smooth $2$-pointed curve $[C, x, y]$ of genus $g-1$ and the normalization map $\nu\colon C\rightarrow C_{xy}$, where $\nu(x)=\nu(y)$, a general point of the irreducible divisor $\Delta_0'$ (respectively of $\Delta_0''$) corresponds to a stable Prym curve $[C_{xy}, \eta]$, where $\eta\in \mathrm{Pic}^0(C_{xy})[2]$ and $\nu^*(\eta)\in \mathrm{Pic}^0(C)$ is non-trivial
(respectively trivial). A general point of $\Delta_{0}^{\mathrm{ram}}$ is the Prym curve $[Y, \eta]$, where $Y:=C\cup_{\{x, y\}} \PP^1$ is a quasi-stable curve  and  $\eta\in \mathrm{Pic}^0(Y)$ satisfies $\eta_{\PP^1}\cong \OO_{\PP^1}(1)$ and $\eta_C^{\otimes 2}\cong \OO_C(-x-y)$. 
Then $\delta_0^{'}:=[\Delta_0']$, $\delta_0^{''}:=[\Delta_0'']$, respectively
 $\delta_0^{\mathrm{ram}}:=[\Delta_0^{\mathrm{ram}}]$. For details we refer to \cite[1]{FL}.

\vskip 4pt

We now let $[X, \mathfrak{N}\oplus_{\perp} \mathbb Z\cdot L]\in \F_g^{\mathfrak{N}}$ be a general standard Nikulin surface of genus $g$. A Lefschetz pencil in the linear system $|L|$ on $X$ induces a family of stable Prym curves  $$\Xi^{\mathrm{s}}:=\Bigl\{[C_t, \mathfrak{e}_{C_t}^{\vee}]: t\in \PP^1\Bigr\}\subseteq \rr_g.$$ Observe that in this pencil there exists precisely one curve having $N_j$ as one of its components, for $j=1, \ldots, 8$. In this case, the corresponding curve in $\Xi^s$ looks like $C_j'+N_j$, where $C_j'$ is a smooth curve of genus $g-1$ and $C_j'\cdot N_j=2$. Furthermore, $\mathfrak{e}^{\vee}_{N_j}\cong \OO_{N_j}(1)$, while $\mbox{deg}\bigl(\mathfrak{e}^{\vee}_{C_j'}\bigr)=-1$. These eight points correspond to the intersection of $\Xi^s$ with the divisor $\Delta_0^{\mathrm{ram}}$. For a general choice of $\Xi^{\mathrm{s}}$, the intersection with $\Delta_0^{\mathrm{ram}}$ is transversal, therefore $\Xi^s\cdot \delta_0^{\mathrm{ram}}=8$, see also \cite[Proposition 1.4]{FV1} for details. Furthermore, $\Xi^s\cdot \lambda=g+1$, $\Xi^s\cdot \delta_0''=0$ and accordingly $$\Xi^s\cdot \delta_0'=\bigl(\pi_*(\Xi^{\mathrm{s}})\bigr)\cdot \delta_0-2\cdot 8=6g+18-2\cdot 8=6g+2.$$ 

\vskip 4pt

Let us now assume that $X$ is a non-standard Nikulin surface of genus $4g-5$ such that $\mbox{Pic}(X)\cong \mathbb Z\cdot R\oplus \mathfrak{N}$, where as above 
$$R\equiv \frac{L-N_1-N_2}{2}\in \mbox{Pic}(X).$$
We choose a Lefschetz pencil $\{C_t\}_{t\in \PP^1}$ in the linear system $|R|$, thus the genus of $C_t$ equals $g-1$. Let 
$\Xi\subseteq \mm_{g-1}$ be the induced curve in moduli. and we denote by $x_{j,t}$ the point of intersection of $C_t$ with $N_j$ for $j=1,2$. We denote by $\Xi^{\mathrm{ns}}\subseteq \Delta_0^{\mathrm{ram}}\subseteq \rr_g$ the family of stable Prym curves 
$$\Xi^{\mathrm{ns}}:=\Bigl\{\bigl[C_t\cup_{\{x_{1,t}, x_{2,t}\}} \PP^1, \eta_{C_t}=\mathfrak{e}_{C_t}^{\vee}, \ \eta_{\PP^1}\cong \OO_{\PP^1}(1)\bigr]:t\in \PP^1\Bigr\}.$$

The intersection numbers of $\Xi^{\mathrm{ns}}$ with the generators of $\mbox{Pic}(\rr_g)$ are as follows:

\vskip 4pt

\begin{proposition}\label{prop:intnumb}
One has $\Xi^{\mathrm{ns}}\cdot \lambda=g, \ \Xi^{\mathrm{ns}}\cdot \delta_0^{\mathrm{ram}}=4,  \
\Xi^{\mathrm{ns}}\cdot \delta_0''=0$ \ and \ $\Xi^{\mathrm{ns}}\cdot \delta_0'=6g$.
\end{proposition}
\begin{proof}
Clearly $\Xi^{\mathrm{ns}}\cdot \lambda=\bigl(\Xi\cdot \lambda\bigr)_{\mm_{g-1}}=g$. Furthermore, $\bigl(\Xi\cdot \delta_0\bigr)_{\mm_{g-1}}=6(g+2)$. The pencil $\Xi$ contains \emph{six} reducible curves of type $N_j+C_j'$ for $j=3, \ldots, 8$, where $C_j'$ is a smooth curve of genus $g-2$ intersecting $N_j$ at two points and having no further intersection with the remaining curve $N_{\ell}$, with 
$\ell\neq j$. These six points will clearly contribute to the intersection $\Xi^{\mathrm{ns}}\cdot \delta_0^{\mathrm{ram}}$. Since $\Xi^{\mathrm{ns}}\cdot \delta_0''=0$, it follows that $\Xi^{\mathrm{ns}}\cdot \delta_0'=6(g+2)-2\cdot 6=6g$.  Finally, from the adjunction formula, we first observe that the contribution to the intersection number ${\pi_*(\Xi^{\mathrm{ns}})}\cdot \delta_0$ coming from $\delta_0^{\mathrm{ram}}$ equals $N_1^2+N_2^2+\#\bigl\{\mbox{nodes in } C_3'+N_3, \ldots, C_8'+N_8\bigr\} =-4+2\cdot 6$. Since $\pi^*(\delta_0)=2\delta_0^{\mathrm{ram}}+\cdots$, we then compute
$$\Xi^{\mathrm{ns}}\cdot\delta_0^{\mathrm{ram}}=6+\frac{1}{2}\bigl(N_1^2+N_2^2\bigr)=6-\frac{4}{2}=4,$$
which finishes the proof.
\end{proof}

We now explain how Theorem \ref{thm:disjoint1} and \ref{thm:disjoint2} determines the part of the class of the divisor $\dd_{2i+2}$ that is relevant to birational geometry applications for $\rr_{2i+1}$. This provides an alternative proof of Theorem 0.2 from \cite{FL}.

\begin{proposition}\label{int:empty1}
One has $\Xi^{\mathrm{s}}\cap \dd_{2i+1}=\emptyset$. Accordingly, also \ $\Xi^{\mathrm{s}}\cdot \dd_{2i+1}=0$.
\end{proposition}
\begin{proof}
We have already explained in the course of proving Theorem \ref{thm:disjoint1} that no irreducible Prym curve in the pencil $\Xi^{\mathrm{s}}$ lies in $\dd_{2i+1}$. It remains to show that also the reducible Prym curves $\bigl[C_j'+N_j, \ \mathfrak{e}_{C_j'}^{\vee}, \ \OO_{N_j}(1)\bigr]\in \Delta^{\mathrm{ram}}_0$ lie outside $\dd_{2i+1}$. But this is precisely the conclusion of Remark \ref{rmk:ns3}.    
\end{proof}

\begin{proposition}\label{int:empty2} One has $\Xi^{\mathrm{ns}}\cap \dd_{2i+1}=\emptyset$. Accordingly,  also 
\ $\Xi^{\mathrm{ns}}\cdot \dd_{2i+1}=0$.
\end{proposition}
\begin{proof}
In the course of the proof of Theorem \ref{thm:disjoint2} we have established that no Prym curve $\xi(C)$ in $\Xi^{\mathrm{ns}}$, where $C$ is irreducible, lies in $\dd_{2i+1}$. It remains to extend this conclusion to the six reducible curves $C=C_j'+N_j$, where $C_j'\in |R-N_j|$ and $j=3, \ldots, 8$. This amounts either to an inspection that each step of the proof of Theorem \ref{thm:disjoint2} also works when $C=C_j+N_j$. Alternatively, one can carry out the steps in Remark \ref{rmk:ns3}. 
\end{proof}

\vskip 4pt

\begin{remark}
    We point out that Propositions \ref{int:empty1} and \ref{int:empty2} determine up to a constant the $\lambda, \delta_0'$ and $\delta_0^{\mathrm{ram}}$-coefficients of $[\dd_{2i+1}]$. Even though in this paper we stop short of computing the class of the diviors $\dd_g^f$, we point out that for $f\geq 2$ it is no longer the case $\Xi^{\mathrm{ns}}\cdot \dd^f_g=0.$ Indeed, setting $e=g$, then equation \ref{eq:expdim-1} implies $g=f^2$ and in this case $\cD_g^f$ is the locus of pairs $[C, \eta]\in \cR_g$ such that $h^0(C, A\otimes \eta)\neq 0$, for one of the linear systems $A\in W_{f^2-1}^{f-1}(C)$. The class of $[\dd_g^f]$ is computed in \cite[Theorem 0.4]{FL} and by direct calculation we observe that $\Xi^{\mathrm{ns}}\cdot \dd_g^f\neq 0$ (while, $\Xi^{\mathrm{s}}\cdot \dd_g^f=0$). Note that when $g=9$ and $f=3$, the divisor $\dd_9^3$ is instrumental in showing that $\rr_9$ in uniruled, see \cite{FV3}.
\end{remark}

\vskip 5pt

\subsection{Open questions} We end the paper with several questions. Using in an essential way the work of Arbarello-Bruno-Sernesi \cite{ABS}  describing certain rational surfaces of Du Val type (that is, blow-ups of $\PP^2$ at $9$ points) as limits of polarized $K3$ surfaces, explicit Brill-Noether general curves of every genus $g$ defined over $\mathbb Q$ have been constructed in \cite{ABFS}.

\begin{question}
Can one construct Prym-canonical curves $[C, \eta]\in \cR_g$ defined over $\mathbb Q$ that are Brill-Noether generic in Prym sense, that is, satisfy $$\mbox{dim } V_e^{e-f}(\omega_C\otimes \eta)=e-f(g-1-e+f)$$
for every $0\leq f<e< g$? Is there an analogue of \emph{Du Val} surfaces in Nikulin setting, that is, a class of explicit rational surfaces of degree $2g-2$ in $\PP^g$ that are limits of general polarized standard Nikulin surfaces $X\stackrel{|L|}\hookrightarrow \PP^g$? The same question can be asked for non-standard Nikulin surfaces.
\end{question}

\begin{question}
Because of the low slope of the $\lambda$ and $\delta_0'$-coefficients of the class $[\dd_{2i+1}]$, this divisor has been instrumental in showing that $\rr_g$ is of general type for odd $g\geq 13$, see \cite{Br}, \cite{FJP}, \cite{FL}. Is the class $[\dd_{2i+1}]$ an extremal point in the effective cone of divisors $\mbox{Eff}(\rr_{2i+1})$? If so, is there a (modular) birational model of $\rr_{2i+1}$ in which the divisor $\dd_{2i+1}$ gets contracted?    
\end{question}

\vskip 4pt

The last question is of a more speculative nature. Let us denote by $\chi \colon \rr_g\rightarrow \mm_{2g-1}$ the map assigning to a Prym curve $[C, \eta]$ the source curve $[\widetilde{C}]$ of the double cover $f\colon \widetilde{C}\rightarrow C$ induced by $\eta$. The pull-back map $\chi^* \colon \mbox{Pic}(\mm_{2g-1})\rightarrow \mbox{Pic}(\rr_g)$ has been described in \cite[Proposition 4.1]{FL}. Setting $g=2i+1$ and recalling that $\mm_{2g-1,g}^1$ is the corresponding Hurwitz divisor on $\mm_{2g-1}$ of curves having gonality at most $g$, putting together results from \cite{HM} and \cite[Theorem 0.2]{FL}, we observe that has the following equality of divisor classes on $\rr_{2i+1}$:
\begin{equation}\label{eq:srange}
\chi^*\Bigl([\mm_{2g-1,g}^1]\Bigr)-\frac{\binom{4i}{2i-1}}{\binom{2i}{i-1}}\cdot \frac{i-1}{4i-1}\cdot \pi^*\Bigl([\mm_{2i+1,i+1}^1]\Bigr)-\frac{\binom{4i}{2i-1}}{\binom{2i}{i}}\cdot \frac{3}{4i-1}\cdot \bigl[\dd_{2i+1}\bigr]\in \mathbb Q\Bigl\langle \delta_j, \delta_{j, g-j}: j\geq 1\Bigr\rangle.
\end{equation}
In other words, in the $\mathbb Q\langle \lambda, \delta_0', \delta_0'', \delta_0^{\mathrm{ram}}\rangle$-subspace of $\mbox{Pic}(\rr_{2i+1})$ an effective linear combination of $\pi^*(\mm_{2i+1,i+1})$ and $\dd_{2i+1}$ is linearly equivalent to the pullback of the Hurwitz divisor $\mm_{2g-1,g}^1$ under the map $\chi$. This invites the following question:

\begin{question}
Fix $i\geq 1$. Does one have a set-theoretic equality on $\cR_{2i+1}$
$$\chi^{-1}\bigl(\cM_{2g-1,g}^1\bigr)=\pi^{-1}(\cM_{2i+1,i+1}) \cup \cD_{2i+1}?$$
Equivalently, given a smooth curve $C$ of genus $g=2i+1$ and of maximal gonality $i+2$, for an \'etale double cover  $f\colon \widetilde{C}\rightarrow C$ such that $f_*\OO_{\widetilde{C}}\cong \OO_C\oplus \eta$, is the following equivalence true:
\begin{equation}\label{eq:speculation}
\mbox{gon}(\widetilde{C})\leq 2i+1\Longleftrightarrow \eta\in C_i-C_i?
\end{equation}
Note that one implication is obviously true. If $\eta\cong \OO_C(D-D')\in C_i-C_i$, with $D$ and $D'$ being effective divisors of degree $i$, then $h^0\bigl(\widetilde{C}, f^*(\OO_C(D))\bigr)=h^0(C, \OO_C(D))+h^0(C, \OO_C(D'))\geq 2$, therefore $\mbox{gon}(\widetilde{C})\leq 2i$. In particular, a positive answer to (\ref{eq:speculation}) would imply that $\mbox{gon}(\widetilde{C})$ always has to be even! Note that a positive answer to Question (\ref{eq:speculation}) for $g=3,5$ is provided in \cite[Theorems 5.1 and 5.2]{FL}. Other questions concerning the Brill-Noether theory of the curve $\widetilde{C}$ have been investigated in \cite{DLC}.
\end{question}

\end{document}